\documentclass[11pt]{amsart}
\usepackage{amsfonts,amssymb,amsmath,amsthm}
\usepackage{url}
\usepackage{enumerate}
\usepackage{latexsym}
\usepackage{multicol}
\usepackage{multirow}
\input xy
\xyoption{all}
\usepackage{hyperref}
\hypersetup{colorlinks,linkcolor=blue ,citecolor=blue, urlcolor=blue}

\def\NZQ{\mathbb}               % the font for N,Z,Q,R,C

\def\ZZ{{\NZQ Z}}
\def\RR{{\NZQ R}}
\def\CC{{\NZQ C}}

\def\frk{\mathfrak}               % font for "Fraktur"

\def\Phi{{\frk N}}
\def\ab{{\bold a}}

\def\eb{{\bold e}}

\def\vb{{\bold v}}
\def\ub{{\bold u}}
\def\wb{{\bold w}}
\def\xb{{\bold x}}

\def\opn#1#2{\def#1{\operatorname{#2}}} % to make operators
\opn\chara{char} 
\opn\length{\ell} 
\opn\pd{pd} 
\opn\rk{rk}
\opn\projdim{proj\,dim} 
\opn\injdim{inj\,dim} 
\opn\rank{rank}
\opn\depth{depth} 
\opn\grade{grade} 
\opn\height{height}
\opn\embdim{emb\,dim} 
\opn\codim{codim}

\opn\Tr{Tr} 
\opn\bigrank{big\,rank}
\opn\superheight{superheight}
\opn\lcm{lcm}
\opn\trdeg{tr\,deg}%\emph{
\opn\reg{reg} 
\opn\lreg{lreg} 
\opn\ini{in} 
\opn\lpd{lpd}
\opn\size{size}
\opn\mult{mult}
\opn\dist{dist}
\opn\cone{cone}
\opn\lex{lex}
\opn\rev{rev}
\opn\div{div} \opn\Div{Div} \opn\cl{cl} \opn\Cl{Cl}
\opn\Spec{Spec} \opn\Supp{Supp} \opn\supp{supp} \opn\Sing{Sing}
\opn\Ass{Ass} \opn\Min{Min}
\opn\Ann{Ann} \opn\Rad{Rad} \opn\Soc{Soc}
\opn\Syz{Syz} \opn\Im{Im} \opn\Ker{Ker} \opn\Coker{Coker}
\opn\Am{Am} \opn\Hom{Hom} \opn\Tor{Tor} \opn\Ext{Ext}
\opn\End{End} \opn\Aut{Aut} \opn\id{id} \opn\ini{in}

\opn\nat{nat}
\opn\pff{pf}%   \pf exists already
\opn\Pf{Pf} \opn\GL{GL} \opn\SL{SL} \opn\mod{mod} \opn\ord{ord}
\opn\Gin{Gin}
\opn\Hilb{Hilb}\opn\adeg{adeg}\opn\std{std}\opn\ip{infpt}
\opn\Pol{Pol}
\opn\sat{sat}
\opn\Var{Var}
\opn\Gen{Gen}

\opn\aff{aff} \opn\con{conv} \opn\relint{relint} \opn\st{st}
\opn\lk{lk} \opn\cn{cn} \opn\core{core} \opn\vol{vol}
\opn\link{link} \opn\star{star}
\opn\gr{gr}

\def\Ac{{\mathcal A}}
\def\Bc{{\mathcal B}}

\def\Hc{{\mathcal H}}
\def\Sc{{\mathcal S}}

\def\Pc{{\mathcal P}}
\def\Qc{{\mathcal Q}}
\def\Rc{{\mathcal R}}

\def\Vol{{\textnormal{Vol}}}
\def\vol{{\textnormal{vol}}}
\def\conv{{\textnormal{conv}}}

\def\ord{{\textnormal{ord}}}

\newtheorem{Theorem}{Theorem}[section]
\newtheorem{Lemma}[Theorem]{Lemma}

\numberwithin{equation}{section}
\newtheorem*{acknowledgement}{Acknowledgment}

\begin{document}
	\title{Classification of lattice polytopes with small volumes}
	\author[T.~Hibi]{Takayuki Hibi}
	\address[Takayuki Hibi]{Department of Pure and Applied Mathematics,
		Graduate School of Information Science and Technology,
		Osaka University,
		Suita, Osaka 565-0871, Japan}
	\email{hibi@math.sci.osaka-u.ac.jp}
	\author[A.~Tsuchiya]{Akiyoshi Tsuchiya}
\address[Akiyoshi Tsuchiya]{
	Graduate school of Mathematical Sciences,
	University of Tokyo,
	Komaba, Meguro-ku, Tokyo 153-8914, Japan} 
\email{akiyoshi@ms.u-tokyo.ac.jp}
	%\date{}
	%\thanks{The second author was partially supported by Grant-in-Aid for JSPS Fellows 16J01549.}
	\begin{abstract}
	In the frame of a classification of general square systems of polynomial equations solvable by radicals, Esterov and Gusev succeeded in classifying all spanning lattice polytopes whose normalized volumes are at most $4$.
	In the present paper, we complete to classify all lattice polytopes whose normalized volumes are at most $4$ based on the known classification of their $\delta$-polynomials.
	\end{abstract} 
	
	\maketitle 
	%\tableofcontents
	%\section*{Background and Results}
	\section{Introduction}
	%\subsection{Background}
	One of the most important, however, unreachable goals of the study on lattice polytopes is to classify all of the lattice polytopes, up to unimodular equivalence.  
	In lower dimension, the following classes of lattice polytopes are classified:
	\begin{itemize}
		\item  $3$-dimensional lattice polytopes  with at most $11$ lattice points \cite{Santos1,Santos2,Santos3};
		\item  $3$-dimensional lattice polytopes with one interior lattice point \cite{int1};
		\item $3$-dimensional lattice polytopes with two interior lattice points \cite{int2}.
	\end{itemize}
	On the other hand, for arbitrary dimension,
	 in each of the following classes of lattice polytopes, a complete classification is known:
	\begin{itemize}
		\item Centrally symmetric smooth Fano polytopes \cite{variety};
		\item Pseudo-symmetric smooth Fano polytopes \cite{class,variety};
		\item Lattice polytopes with $\delta$-binomials \cite{BH1,BH2,BN};
		\item Lattice polytopes with palindromic $\delta$-trinomials \cite{BJ,HNT}.
	\end{itemize}
	It is fashionable among the study on lattice polytopes to classify the lattice polytopes with a given $\delta$-polynomial.
	In the present paper, we will classify all lattice polytopes of arbitrary dimension whose normalized volumes are at most $4$ based on the known classification of their $\delta$-polynomials \cite{HHL,HHN}.
	%Taking into consideration the fact that a complete characterization of the $\delta$-polynomials of lattice polytopes whose normalized volumes are at most $4$ is known (\cite{HHL,HHN}), it is reasonable to classify the lattice polytopes whose normalized volumes are at most $4$.  In fact, in the present paper, this job will be done. 
	In the frame of a classification of general square systems of polynomial
	equations solvable by radicals, Esterov and Gusev \cite{4class}
	succeeded in classifying all lattice polytopes $\Pc \subset \RR^d$ whose normalized
	volumes are at most $4$ for which  $\ZZ((\Pc,1) \cap \ZZ^{d+1})=\ZZ^{d+1}$.
	(Here $\ZZ((\Pc,1) \cap \ZZ^{d+1})=\{z_1\xb_1+\cdots+z_n\xb_n : z_1,\ldots,z_n \in \ZZ \}$ for $(\Pc,1) \cap \ZZ^{d+1}=\{\xb_1,\ldots,\xb_n \} \subset \ZZ^{d+1}$.)
	However, the condition $\ZZ((\Pc,1) \cap \ZZ^{d+1})=\ZZ^{d+1}$ is
	rather strong for achieving a classification of lattice polytopes.
	For example, no empty simplex satisfies the property and, in
	addition, there exists a lattice non-simplex whose normalized volume is $4$ and that lacks the property.
	Combining our work with a result of Esterov and Gusev \cite{4class} will establish a
	complete classification of lattice polytopes whose normalized volumes
	are at most 4 with their $\delta$-polynomials.
	
	\subsection{Possible $\delta$-polynomials}
	
	We recall a complete characterization of the $\delta$-polynomials of lattice polytopes whose normalized volumes are at most $4$.
	
	Let us recall from \cite{BR15} and \cite[Part II]{HibiRedBook} 
	what the {\em $\delta$-polynomial} of a lattice polytope is.
	A {\em lattice polytope} is a convex polytope
	all of whose vertices have integer coordinates. 
	Let $\Pc \subset \RR^d$ be a lattice polytope of dimension $d$
	and define $\delta(\Pc, t)$ by the formula
	\[
	\delta(\Pc, t)
	= (1 - t)^{d+1} \left[1 + \sum_{n=1}^{\infty} |n\Pc \cap \ZZ^d| t^n\right],
	\]
	where $n\Pc=\{n \ab : \ab \in \Pc \}$, the dilated 
	polytopes of $\Pc$.
	Then it is known that $\delta(\Pc, t)$ is a polynomial in $t$ 
	of degree at most $d$.
	We say that the polynomial $\delta(\Pc, t)=\delta_0+\delta_1t+\cdots+\delta_dt^d$
	is the {\em $\delta$-polynomial}  (or the {\em $h^*$-polynomial}) of $\Pc$ and  the sequence $\delta(\Pc)=(\delta_0,\delta_1,\ldots,\delta_d)$ is the {\em $\delta$-vector} (or the {\em $h^*$-vector}) of $\Pc$.
	The following properties of $\delta(\Pc)$ are known:
	\begin{itemize}
		\item $\delta_0=1$, $\delta_1=|\Pc \cap \ZZ^d|-(d+1)$ and $\delta_d=|\text{int}(\Pc) \cap \ZZ^d|$, where $\text{int} (\Pc)$ is the interior of $\Pc$. Hence one has $\delta_1 \geq \delta_d$;
		\item $\delta_i \geq 0$ for each $i$ \cite{Stanleynonnegative};
		\item When $\delta_d \neq 0$, one has $\delta_i \geq \delta_1$ for $1 \leq i \leq d-1$  \cite{Hibi_ineq};
		\item  $\sum_{i=0}^{d}\delta_i/d!$ coincides with the usual volume of $\Pc$ \cite[Proposition 4.6.30]{StanleyEC1}. In general, the positive integer $\sum_{i=0}^{d}\delta_i $ is said to be the \textit{normalized volume} of $\Pc$, denoted by $\Vol(\Pc)$.
	\end{itemize}

	%Let $\Pc \subset \RR^d$ be a lattice polytope of dimension $d$ and $(\delta_0,\ldots,\delta_d)$ the $\delta$-vector of $\Pc$.
	
	In  \cite[Theorem 5.1]{HHL} and \cite[Theorem 0.1]{HHN}, the possible $\delta$-polynomials with $\delta_0+\cdots+\delta_d \leq 4$ are completely classified.
	%\begin{Theorem}[{\cite[Theorem 0.1]{HHN}}]
	%	Let $(\delta_0,\ldots,\delta_d)$ be a sequence of nonnegative integers, where $\delta_0=1$ and $\delta_1 \geq \delta_d$, which satisfies $\delta_0+\cdots+\delta_d \leq 3$.
	%	Then there exists a lattice polytope of dimension $d$ whose $\delta$-vector coincides with $(\delta_0,\ldots,\delta_d)$
	%if and only if $(\delta_0,\ldots,\delta_d)$ satisfies all inequalities {\em(\ref{eq1})} and {\em (\ref{eq2})}.
	%\end{Theorem}
	\begin{Theorem}[\cite{HHL,HHN}]
	\label{delta4}
	Let $2 \leq V \leq 4$ be a positive integer and  $1+t^{i_1}+\cdots+t^{i_{V-1}}$ a polynomial with $1 \leq i_1 \leq \cdots \leq i_{V-1} \leq d$.
	Then there exists a lattice polytope of dimension $d$ whose $\delta$-polynomial equals
	$1+t^{i_1}+\cdots+t^{i_{V-1}}$ if and only if one of the following is satisfied:
	\begin{enumerate}
		\item[{\rm (1)}] $V = 2$ and $i_1 \leq \lfloor(d+1)/2 \rfloor$;\\
		\item[{\rm (2)}] $V = 3$, $2i_1 \geq i_2$ and  $i_1+i_2  \leq d+1$;\\
		\item[{\rm (3)}] $V = 4$, $i_3 \leq i_1 +i_2$, $i_1+i_3 \leq d+1$ and $i_2 \leq \lfloor (d+1)/2 \rfloor$, and the additional condition
		$$2i_2 \leq i_1+i_3 \ \textnormal{or} \ i_2+i_3 \leq d+1.$$ 
	\end{enumerate}
\end{Theorem}
	We remark that when $\delta_0+\cdots+\delta_d \leq 4$, all the possible $\delta$-polynomials can be obtained by lattice simplices.
	However, when $\delta_0+\cdots+\delta_d=5$, this is not true \cite[Remark 5.3]{HHL}.
	Most recently, the possible $\delta$-polynomials with $\delta_0+\cdots+\delta_d=5$ are completely classified \cite{Hig_prime, delta5}.
	
	\subsection{Main results}
	Recall that a matrix $A \in \ZZ^{d \times d}$ is {\em unimodular} if $\det (A) = \pm 1$.
	Given lattice polytopes $\Pc$ and $\Qc$ in $\RR^d$ of dimension $d$,
	we say that $\Pc$ and $\Qc$ are {\em unimodularly equivalent}
	if there exist a unimodular matrix $U \in \ZZ^{d \times d}$
	and a lattice point $\wb$, such that $\Qc=f_U(\Pc)+\wb$,
	where $f_U$ is the linear transformation in $\RR^d$ defined by $U$,
	i.e., $f_U({\bf v}) = {\bf v} U$ for all ${\bf v} \in \RR^d$.
	
	For a lattice polytope $\mathcal{P} \subset \RR^d$ of dimension $d$, the {\em lattice pyramid} over $\mathcal{P}$ is defined by $\text{conv}(\mathcal{P}\times \left\{ 0 \right\} ,(0,\ldots,0,1))$ $\subset \RR^{d+1}$. 
	Let ${\rm Pyr}(\Pc)$ denote this polytope. 
	We often use the term lattice pyramid  for a lattice polytope that has been obtained by successively taking lattice pyramids.
	Note that the $\delta$-polynomial does not change under lattice pyramids \cite{Bat}.
	Therefore, it is essential that we classify lattice polytopes which are not lattice pyramids over any lower-dimensional lattice polytope.

	A lattice polytope $\Pc \subset \RR^d$ for which 
	$\ZZ((\Pc,1) \cap \ZZ^{d+1})=\ZZ^{d+1}$  is called \textit{spanning}.
	From the work of Hofscheier, Katth\"an and Nill it follows that there are only finitely many lattice spanning polytopes $\Pc \subset \RR^d$ of given normalized volume (and arbitrary dimension) up to unimodular equivalence and lattice pyramid constructions \cite[Corollary 2.4]{HKN}.
	In particular, Esterov and Gusev \cite{4class} gave only finitely many lattice spanning polytopes whose normalized volumes are at most $4$ for their classification result.
	However, it is hard to classify lattice non-spanning polytopes. 
	In fact, there exist infinitely many lattice non-spanning polytopes of given normalized volume even up to unimodular equivalence and lattice pyramid constructions. 
	
	In the present paper, we will complete to
	 classify, up to unimodular equivalence and lattice pyramid constructions, the lattice polytopes whose normalized volumes are at most $4$.
	The complete classification of the lattice polytopes whose normalized volumes are at most $4$ up to unimodular equivalence consists of these polytopes and lattice pyramids over them.
	Note that every lattice simplex of dimension $d$ with $\Vol(\Pc)=1$ is unimodularly equivalent to the standard simplex of dimension $d$. And the normalized volume of a lattice non-spanning non-simplex is at least $4$.
	In order to classify all lattice polytopes whose normalized volumes are at most $4$, we should consider the following three cases:
	\begin{enumerate}
		\item Lattice simplices $\Delta \subset \RR^d$ with $\Vol(\Delta) \leq 4$;
		\item Lattice spanning non-simplices $\Pc \subset \RR^d$ with $\Vol(\Pc) \leq  4$;
		\item Lattice non-spanning non-simplices $\Pc \subset \RR^d$ with $\Vol(\Pc) = 4$.
	\end{enumerate}
	The complete classification of the case (2) can be obtained from \cite{4class}. Therefore, we will show the cases (1) and (3).
	
	Let ${\mathbf 0}$ denote the origin of $\RR^d$ and let $\eb_1,\ldots,\eb_d$ denote the canonical unit coordinate vectors of $\RR^d$.
	First, the complete classification of the lattice simplices $\Delta \subset \RR^d$ 
	with $\Vol(\Delta) \leq 4$ can be obtained from the following:
	\begin{Theorem}
		\label{simplex}
		Let $\Delta \subset \RR^d$ be a lattice simplex of dimension $d$ whose $\delta$-polynomial equals $1+t^{i_1}+\cdots+t^{i_{V-1}}$ with $2 \leq V \leq 4$ and $1 \leq i_1 \leq \cdots \leq i_{V-1}$.
		Assume that $\Delta$ is not a lattice pyramid.
		Then there exist, up to unimodular equivalence, exactly the following $5$ possibilities for $\Delta$:
		\begin{enumerate}
			\item[{\rm (1)}] $V=2:\Delta^{(2)};$
			\item[{\rm (2)}] $V=3:\Delta^{(3)};$
			\item[{\rm (3)}] $V=4:\Delta^{(4)}_i$, $1 \leq i \leq 3.$
		\end{enumerate}
		The conditions and vertices of $\Delta$ are presented in  {\rm TABLE \ref{table:simplex}}.
	
		\begin{table}[h]
			%\centering
			\hspace*{-15mm}
			\begin{tabular}{l|l|l}
				& {\rm conditions}                 & \multicolumn{1}{l}{{\rm vertices}}                                                                                                                                        \\ \hline
				$\Delta^{(2)}$                                    & $2i_1=d+1$                          & \multicolumn{1}{l}{${\bf 0}, \eb_1,\ldots,\eb_{d-1}, \eb_1+\cdots+\eb_{d-1}+2\eb_d$}                                                                                \\ \hline
				\multirow{2}{*}{$\Delta^{(3)}$}   & \multirow{2}{*}{$2i_1 \geq i_2$  }     & ${\bf 0}, \eb_1,\ldots,\eb_{d-1}$,                                                                                                                                   \\
				&   $i_1+i_2=d+1,$                                                     & $2\sum\limits_{i=1,i \neq d}^{-i_1+2i_2-1}\eb_i+\sum\limits_{i=-i_1+2i_2}^{d-1}\eb_i+3\eb_d$                                                                         \\ \hline
				\multirow{3}{*}{$\Delta^{(4)}_1$}  & $i_1<i_2<i_3$,                      & ${\bf 0}, \eb_1,\ldots,\eb_{d-1}$,                                                                                                            \\ 
				& $i_3 \leq i_1+i_2$,    & $\sum\limits_{i=1}^{2i_1-i_2}\eb_i+3\sum\limits_{i=2i_1-i_2+1}^{d-1}\eb_i+4\eb_d$ \\ 
				& $2i_2=i_1+i_3=d+1$                      &                                                                                                                                            \\ \hline
				\multirow{2}{*}{$\Delta^{(4)}_2$} & \multirow{2}{*}{$i_3 \leq i_1+i_2$,}     & ${\bf 0}, \eb_1,\ldots,\eb_{d-1}$,                                                                                                                                   \\
				&$i_2+i_3=d+1$                                    & $2\sum\limits_{i=1}^{d-2i_1+1}\eb_i+\sum\limits_{i=d-2i_1+2,i \neq d}^{-i_1+2i_2}\eb_i+3\sum\limits_{i=-i_1+2i_2+1}^{d-1}\eb_i+4\eb_d$                   \\ \hline
				\multirow{2}{*}{$\Delta^{(4)}_3$} &  \multirow{2}{*}{$i_3 \leq i_1+i_2$,} & ${\bf 0}, \eb_1,\ldots, \eb_{d-2}$, \\                                                                                                         
				&$i_1+i_2+i_3=d+1$                                    & $\sum\limits_{i=-i_1+i_2+i_3}^{d-2}\eb_i+2\eb_{d-1},\sum\limits_{i=1}^{-i_1+i_2+i_3-1}\eb_i+\sum\limits_{i=2i_3-1}^{d-2}\eb_i+2\eb_d$   
			\end{tabular}
			\bigskip
			\caption{The lattice simplices $\Delta \subset \RR^d$ with $\Vol(\Delta) \leq 4$ in Theorem \ref{simplex}.}
			\label{table:simplex}
		\end{table}
	\end{Theorem}
	Second, the complete classification of the lattice spanning non-simplices $\Pc \subset \RR^d$ with $\Vol(\Pc) \leq 4$ can be obtained from the following:
	\begin{Theorem}[\cite{4class}]
		\label{spanning}
		Let $2 \leq V \leq 4$ be a positive integer and  $\Pc \subset \RR^d$ a lattice spanning non-simplex with $\Vol(\Pc)=V$.
		Assume that $\Pc$ is not a lattice pyramid.
		Then there exist up to unimodular equivalence exactly the following $24$ possibilities for $\Pc$:
		\begin{enumerate}
			\item[{\rm (1)}] $\delta(\Pc,t)=1+t:\Pc^{(2)};$
			\item[{\rm (2)}] $\delta(\Pc,t)=1+2t:\Pc^{(3)}_i, 1 \leq i \leq 2;$
			\item[{\rm (3)}] $\delta(\Pc,t)=1+t+t^2:\Qc^{(3)}_i$, $1 \leq i \leq 2;$
			\item[{\rm (4)}] $\delta(\Pc,t)=1+3t:\Pc^{(4)}_i, 1 \leq i \leq 4;$
			\item[{\rm (5)}] $\delta(\Pc,t)=1+2t+t^2:\Qc^{(4)}_i$, $1 \leq i \leq 9;$
			\item[{\rm (6)}] $\delta(\Pc,t)=1+t+2t^2:\Rc^{(4)}_i, 1 \leq i \leq 2;$
			\item[{\rm (7)}] $\delta(\Pc,t)=1+t+t^2+t^3:\Sc^{(4)}_i$, $1 \leq i \leq 4;$
		\end{enumerate}
		The dimension, vertices, and the $f$-vector of $\Pc$ are presented in {\rm TABLE  \ref{table:span}}.
		\begin{table}[h]
			\centering 
			\begin{tabular}{l|l|l|l}
				& $d$ & {\rm vertices} & $f$-{\rm vector} \\ \hline
				$\Pc^{(2)}$	& $2$    & ${\mathbf 0},\eb_1,\eb_2,\eb_1+\eb_2$  & $(1,4,4)$       \\ \hline
				$\Pc^{(3)}_1$		&  $2$   &   ${\bf 0},2\eb_1,\eb_2,\eb_1+\eb_2$    & $(1,4,4)$   \\ \hline
				$\Pc^{(3)}_2$		&  $3$   &     ${\bf 0},\eb_1,\eb_2,\eb_3,\eb_1+\eb_3,\eb_2+\eb_3$   & $(1,6,9,5)$  \\ \hline
				$\Qc^{(3)}_1$		&   $3$  &       	  ${\bf 0},\eb_1,\eb_2,\eb_3,\eb_1+\eb_2-2\eb_3 $ & $(1,5,9,6)$   \\ \hline
				$\Qc^{(3)}_2$		&   $4$  &      ${\mathbf 0}, \eb_1,\eb_2,\eb_3,\eb_4,-\eb_1-\eb_2+\eb_3+\eb_4$  & $(1,6,15,18,9)$  \\ \hline
				$\Pc^{(4)}_1$	& $2$     &   ${\bf 0},2\eb_1,\eb_2,2\eb_1+\eb_2$ & $(1,4,4)$
				\\ \hline
				$\Pc^{(4)}_2$	&   $2$  &    ${\bf 0}, 3\eb_1,\eb_1+\eb_2,2\eb_1+\eb_2$   & $(1,4,4)$   \\ \hline
				$\Pc^{(4)}_3$	&  $3$   &       ${\bf 0},\eb_1,\eb_2,\eb_1+\eb_3,\eb_2+\eb_3,2\eb_3$  & $(1,6,9,5)$  \\ \hline
				$\Pc^{(4)}_4$	&   $4$  &       ${\bf 0},\eb_1,\eb_2,\eb_3,\eb_4,\eb_1+\eb_2,\eb_1+\eb_3,\eb_1+\eb_4$  & $(1,8,16,14,6)$   \\ \hline
				$\Qc^{(4)}_1$	&   $2$  &   $\eb_1,-\eb_2,\eb_1-\eb_2,-\eb_1+\eb_2$ & $(1,4,4)$
				\\ \hline
				$\Qc^{(4)}_2$	&   $2$  &   $\eb_1,\eb_2,-\eb_1,-\eb_2$   & $(1,4,4)$    \\ \hline
				$\Qc^{(4)}_3$	&   $3$  & $\eb_1,\eb_2,\eb_3,\eb_1+\eb_2,-\eb_3 $  & $(1,5,9,6)$      \\ \hline
				$\Qc^{(4)}_4$	&   $3$  &   ${\mathbf 0}, \eb_1,\eb_2,\eb_1+\eb_2, 2\eb_3$  & $(1,5,8,5)$      \\ \hline
				$\Qc^{(4)}_5$	&   $3$  &   ${\bf 0}, \eb_1,\eb_2,\eb_3,\eb_1+\eb_2,\eb_1+\eb_2+\eb_3$    & $(1,6,11,7)$     \\ \hline
				$\Qc^{(4)}_6$	&   $3$  &   ${\bf 0},\eb_1,\eb_2,\eb_3,\eb_1+\eb_2,\eb_1+\eb_2-\eb_3$ & $(1,6,12,8)$ 
				\\ \hline
				$\Qc^{(4)}_7$	&   $4$  &     ${\bf 0}, 2\eb_1,\eb_4,\eb_2+\eb_4,\eb_3+\eb_4,\eb_2+\eb_3+\eb_4$ & $(1,6,13,13,6)$ 
				\\ \hline
				$\Qc^{(4)}_8$	&   $4$  &    	
				${\mathbf 0}, \eb_1,\eb_2,\eb_1+\eb_2,\eb_3,\eb_4,\eb_3+\eb_4$ & $(1,7,17,18,8)$ 
				\\ \hline
				$\Qc^{(4)}_9$	&  $5$   &     ${\bf 0},\eb_1,\eb_2, \eb_1+\eb_2,\eb_5,\eb_3+\eb_5,\eb_4+\eb_5,\eb_3+\eb_4+\eb_5$   & $(1,8,24,34,24,8)$    \\ \hline
				$\Rc^{(4)}_1$	&   $3$  &   ${\bf 0}, \eb_1,\eb_2,\eb_3,\eb_1+\eb_2-3\eb_3$ & $(1,5,9,6)$ 
				\\ \hline
				$\Rc^{(4)}_2$	&   $4$  &   ${\bf 0}, \eb_1,\eb_2,\eb_3, \eb_4, -2\eb_1-\eb_2+\eb_3+\eb_4$    & $(1,6,15,18,9)$      \\ \hline
				$\Sc^{(4)}_1$	&  $4$  &   ${\bf 0}, \eb_1,\eb_2,\eb_3, \eb_4,-\eb_1-\eb_2-\eb_3+\eb_4$ & $(1,6,14,16,8)$
				\\ \hline
				$\Sc^{(4)}_2$	&   $4$  &    ${\bf 0}, \eb_1,\eb_2,\eb_3, \eb_4,-\eb_1-\eb_2-\eb_3+2\eb_4$ & $(1,6,14,16,8)$   \\ \hline
				$\Sc^{(4)}_3$	&  $5$   &     ${\bf 0}, \eb_1,\eb_2,\eb_3, \eb_4,\eb_5,-2\eb_1-\eb_2+\eb_3+\eb_4+\eb_5$  & $(1,7,21,34,30,12)$    \\ \hline
				$\Sc^{(4)}_4$	&   $6$  &    ${\bf 0}, \eb_1,\eb_2,\eb_3, \eb_4,\eb_5,\eb_6,-\eb_1-\eb_2-\eb_3+\eb_4+\eb_5+\eb_6$ & $(1,8,28,56,68,48,16)$ \\  
			\end{tabular}
			\bigskip
			\caption{The lattice spanning non-simplices $\Pc$ with $\Vol(\Pc) \leq 4$ in Theorem \ref{spanning}.}
			\label{table:span}
		\end{table}
		
	\end{Theorem}
	Finally, the complete classification of the lattice non-spanning non-simplices $\Pc \subset \RR^d$ with $\Vol(\Pc) = 4$ can be obtained from the following:
	\begin{Theorem}
		\label{nonspan}
	Let	$\Pc \subset \RR^d$ be a lattice non-spanning non-simplex with $\Vol(\Pc)=4$.
		Assume that $\Pc$ is not a lattice pyramid.
		Then there exist, up to unimodular equivalence, exactly the following $4$ possibilities for $\Pc$:
		\begin{enumerate}
			\item[{\rm (1)}] $\delta(\Pc,t)=1+t+t^k+t^{k+1}$ with $k \geq 2:\Ac^{(4)}_{i}, 1 \leq i \leq 3;$
			\item[{\rm (2)}] $\delta(\Pc,t)=1+t+2t^{k}$ with $k \geq 2:\Bc^{(4)}$.
		\end{enumerate}
		The dimension and vertices of $\Pc$ are presented in  {\rm TABLE \ref{table:nonspan}}.
		In particular, $\Pc$ is (non-unimodularly) equivalent to a pyramid over a square.
		\begin{table}[h]
			\centering
			\begin{tabular}{l|l|l}
				& $d$  & {\rm vertices} \\ \hline
				$\Ac^{(4)}_{1}$	& $2k$  &  ${\mathbf 0}, \eb_1,\ldots,\eb_{d-1},\sum_{j=2}^{d-1}\eb_j+2\eb_d,-\eb_{1}+\eb_{2}$
				\\ \hline
				$\Ac^{(4)}_{2}$	& $2k+1$ &   ${\mathbf 0}, \eb_1,\ldots,\eb_{d-1},\sum_{j=3}^{d-1}\eb_j+2\eb_d,\eb_{1}+\eb_{2}$
				\\ \hline
				$\Ac^{(4)}_{3}$	& $2k+2$ &  ${\mathbf 0}, \eb_1,\ldots,\eb_{d-1},\sum_{j=4}^{d-1}\eb_j+2\eb_d,\eb_{1}+\eb_{2}-\eb_{3}$
				\\ \hline
				$\Bc^{(4)}$	& $2k$ & 	 ${\mathbf 0}, \eb_1,\ldots,\eb_{d-1}, \sum_{j=2}^{d-1}\eb_j+2\eb_d, \eb_1-\eb_{2}$ \\ 
			\end{tabular}
			\bigskip
			\caption{The lattice non-spanning non-simplices $\Pc$ with $\Vol(\Pc) = 4$ in Theorem \ref{nonspan}, where $k \geq 2$.} 
			\label{table:nonspan}
		\end{table}
	\end{Theorem}
	
	The present paper is organized as follows:
	First, in Section \ref{sec1}, we introduce basic materials on lattice polytopes and summarize lemmata which will be indispensable in what follows. We then, in Section \ref{sec2}, prove Theorem \ref{simplex}.
	Finally, in Section \ref{sec3}, we prove Theorem \ref{nonspan}.

	\begin{acknowledgement}{\rm
			The authors would like to thank Gabriele Balletti for pointing out
			minor imperfections in our classification done in our first draft and,
			in addition, for informing them about the paper \cite{4class}.
			The authors would like to thank anonymous referees for their careful readings and many helpful comments which make this paper more readable, in particular, the proof of Theorem \ref{nonspan} was significantly shortened. The first author is partially supported by JSPS KAKENHI 19H00637.
						The second author is partially supported by Grant-in-Aid for JSPS Fellows 16J01549.
		}
	\end{acknowledgement}

	\section{Basic materials on lattice polytopes}
	\label{sec1}
	In this section, we recall basic materials on lattice polytopes and we prepare essential lemmata in this paper.
	First, we introduce the associated finite abelian groups of lattice simplices.
	For a lattice simplex $\Delta \subset \RR^d$ of dimension $d$ whose vertices are $\vb_0,\ldots,\vb_d \in \ZZ^d$,
	set 
	$$\Lambda_\Delta=\left\{(\lambda_0,\ldots,\lambda_d) \in (\RR/\ZZ)^{d+1} : \sum\limits_{i=0}^{d}\lambda_i(\vb_i,1) \in \ZZ^{d+1}   \right\}.$$
	The collection $\Lambda_\Delta$ forms a finite abelian group with addition defined as follows: 
	For $(\lambda_0,\ldots,\lambda_d) \in (\RR/\ZZ)^{d+1}$ and $(\lambda_0',\ldots,\lambda_d') \in (\RR/\ZZ)^{d+1}$,  $(\lambda_0,\ldots,\lambda_d)+(\lambda_0',\ldots,\lambda_d')=(\lambda_0+\lambda_0',\ldots,\lambda_d+\lambda_d') \in (\RR/\ZZ)^{d+1}$.
	We denote the unit of $\Lambda_\Delta$ by $\bold{0}$, and the inverse of $\lambda$ by $-\lambda$,
	and also denote $\underbrace{\lambda+\cdots+\lambda}_{j}$  by $j\lambda$ for an integer $j>0$ and $\lambda \in \Lambda_\Delta$.
	For $\lambda=(\lambda_0,\ldots,\lambda_{d}) \in \Lambda_{\Delta}$, where each $\lambda_i$ is taken with $0 \leq \lambda_i < 1$, we set
	$\textnormal{ht}(\lambda)=\sum_{i=0}^{d}\lambda_i \in \ZZ$,
	and $\textnormal{ord}(\lambda)=\min\{\ell \in \ZZ_{> 0} : \ell\lambda={\bf 0} \}$.
	
	In \cite{BH2}, it is shown that there is a bijection between unimodular equivalence classes of $d$-dimensional lattice simplices with a 
	chosen ordering of their vertices and finite abelian subgroups of $(\RR/\ZZ)^{d+1}$ such that the sum of all entries of each element is an integer.
	In particular, two lattice simplices $\Delta$ and $\Delta'$ are unimodularly equivalent if and only if there exist orderings of their vertices such that $\Lambda_\Delta=\Lambda_{\Delta'}$.
	Moreover, we can characterize lattice pyramids in terms of the associated finite abelian groups by using the following lemma.
	\begin{Lemma}[{\cite[Lemma 12]{Nill}}]
		\label{lem:pyr}
		Let $\Delta \subset \RR^d$ be a lattice simplex of dimension $d$.
		Then $\Delta$ is a lattice pyramid if and only if there is $i \in \{0,\ldots,d\}$ such that $\lambda_i=0$ for all $(\lambda_0,\ldots,\lambda_d) \in \Lambda_\Delta$.
	\end{Lemma}
	
	It is well known that the $\delta$-polynomial of the lattice simplex $\Delta$ can be computed as follows: 
	\begin{Lemma}[{\cite[Proposition 2.6]{BN}}]
		\label{delta}
		Let $\Delta$ be a lattice simplex of dimension $d$ whose $\delta$-polynomial  equals $\delta_0+\delta_1 t+\cdots+\delta_{d}t^d$.
		Then for each $i$, we have $\delta_i=|\{\lambda \in \Lambda_\Delta : \textnormal{ht}(\lambda)=i\}|$.
		In particular, one has $\Vol(\Delta)=|\Lambda_{\Delta}|$.
	\end{Lemma} 
	
	Let $\Pc \subset \RR^d$ be a lattice polytope of dimension $d$.
	Given integers $n = 1, 2, \ldots,$ we define the function $L_{\Pc}(n)$ as follows:
	\[
	L_{\Pc}(n):= \left| n\Pc \cap \ZZ^{d} \right|.
	\]
	Then it is known that $L_{\Pc}(n)$ is a polynomial in $n$ of degree $d$ with $L_{\Pc}(0) = 1$ (see \cite{Ehrhart}). 
	We call $L_{\Pc}(n)$ the {\em Ehrhart polynomial} of $\Pc$.  
	The Ehrhart polynomial $L_{\Pc}(n)$ can be computed by using the $\delta$-vector of $\Pc$.
		\begin{Lemma}
			\label{delta_ehr}
		Let $\Pc \subset \RR^d$ be a lattice polytope of dimension $d$ and 
		$(\delta_0,\ldots,\delta_d)$ the $\delta$-vector of $\Pc$.
		Then one has
		$$L_{\Pc}(n)=\sum_{i=0}^{d}\delta_{i}\binom{n+d-i}{d}.$$
	\end{Lemma}
	
	Finally, we recall a useful technique to compute Ehrhart polynomials. 

	\begin{Lemma}
		\label{delta_split}
		Let $\Pc \subset \RR^d$ be a lattice non-simplex.
		Assume that $\Pc$ has a lattice triangulation which consists of two maximal simplices $\Delta_1$ and $\Delta_2$.
		Then one has $L_\Pc(n)=L_{\Delta_1}(n)+L_{\Delta_2}(n)-L_{\Delta_1 \cap \Delta_2}(n)$.
	\end{Lemma}
	\section{Proof of Theorem \ref{simplex}}
	\label{sec2}
	%For nonnegative integers $a$ and $b$,
	%we let $\Lambda(a,b)$ be the finite abelian subgroups of $(\RR/\ZZ)^{a+b+c}$ defined as follows:
	%$$\Lambda(a,b)= \left\langle \left(\underbrace{\dfrac{1}{3},\ldots,\dfrac{1}{3}}_{a},\underbrace{\dfrac{2}{3},\ldots,\dfrac{2}{3}}_{b}\right) \right \rangle,$$
	%and if for any element of $\Lambda(a,b)$, the sum of all entries of it is an integer, then we let 
	%$\Delta(a,b)$  be a lattice simplex of dimension $a+b-1$ such that $\Lambda_{\Delta(a,b)}=\Lambda(a,b)$. 
	
	%\begin{Remark}
	%	Set $d=a+b-1$ and assume that $a \geq b$.
	%If $a,b \geq 1$, 
	%	Then $\Delta(a,b)$ is unimodularly equivalent to the lattice simplex which is the convex hull of ${\bf 0}, \eb_1,\ldots,\eb_{d-1}$ and $$\sum_{i=1}^{b}\eb_i+2\sum_{i=b+1}^{d-1}\eb_i+3\eb_d.$$
	%and if $b=0$, then $\Delta(a,b)$ is unimodularly equivalent to the lattice simplex which is the convex hull of ${\bf 0}, \eb_1,\ldots,\eb_{d-1}$ and $$2\sum_{i=1}^{d-1}\eb_i+3\eb_d.$$
	%\end{Remark}
	
	%Now, we classify any lattice simplex whose normalized volume equals $3$ up to unimodular equivalence.
	In this section, we classify the lattice simplices $\Delta \subset \RR^d$ of dimension $d$ with $\Vol(\Delta) \leq 4$ up to unimodular equivalence and lattice pyramid constructions. Namely, we prove Theorem \ref{simplex}.
	In order to do this job, we consider the following three cases:
	\begin{enumerate}
		\item $\Vol(\Delta)=2$ (Subsection \ref{sub2.1}); 
		\item $\Vol(\Delta)=3$ (Subsection \ref{sub2.2}); 
		\item $\Vol(\Delta)=4$ (Subsection \ref{sub2.3}).
	\end{enumerate}
	\subsection{The case $\Vol(\Delta)=2$}
	\label{sub2.1}
	In this subsection, we consider the case where $\Vol(\Delta)=2$.
	It follows from Theorem \ref{delta4} (1) that $i_1 \leq \lfloor(d+1)/2 \rfloor$.
	Since $|\Lambda_{\Delta}|=2$, for any $\lambda \in \Lambda_{\Delta}\setminus \{{\mathbf 0}\}$, $\ord(\lambda)=2$. Hence since $\Delta$ is not a lattice pyramid, by using Lemma \ref{lem:pyr}, it follows that  $\Lambda_{\Delta}$ is generated by one element $(1/2,\ldots,1/2)$ and $d+1$ is an even number.
	By using Lemma \ref{delta}, one has $2i_1=d+1$. Moreover it is easy to see that $\Lambda_{\Delta^{(2)}}=\Lambda_{\Delta}$ with any ordering of the vertices of $\Delta^{(2)}$.
	Hence this completes the proof of the case where $\Vol(\Delta)=2$.
	
	\subsection{The case $\Vol(\Delta)=3$}
	\label{sub2.2}
	In this subsection, we consider the case where $\Vol(\Delta)=3$.
	It follows from Theorem \ref{delta4} (2) that $2i_1 \geq i_2$ and $i_1+i_2 =(d+1)/2$.
	For nonnegative integers $a$ and $b$,
	we let $\Lambda(a,b)$ be the finite abelian subgroups of $(\RR/\ZZ)^{a+b}$ defined as follows:
	$$\Lambda(a,b)= \left\langle \left(\underbrace{\dfrac{1}{3},\ldots,\dfrac{1}{3}}_{a},\underbrace{\dfrac{2}{3},\ldots,\dfrac{2}{3}}_{b}\right) \right \rangle.$$
	%and if for any element of $\Lambda(a,b)$, the sum of all entries of the element is an integer, then we let 
	%$\Delta(a,b)$  be a lattice simplex of dimension $a+b-1$ such that $\Lambda_{\Delta(a,b)}=\Lambda(a,b)$ with some ordering of the vertices of $\Delta(a,b)$. 
	Since $\Vol(\Delta)=|\Lambda_{\Delta}|=3$,
	for any $\lambda \in \Lambda_{\Delta}\setminus \{{\mathbf 0}\}$, $\ord(\lambda)=3$. 
	Hence since $\Delta$ is not a lattice pyramid, by Lemma  \ref{lem:pyr}, there exist nonnegative integers $a, b$ with $a+b=d+1$ such that $\Lambda_{\Delta}=\Lambda(a,b)$ with some ordering of the vertices of $\Delta$.
	Since $\Lambda(a,b)$ coincides with $\Lambda(b,a)$ by reordering of the coordinates, we can assume that $a \geq b$.
	Then by using Lemma \ref{delta}, one has $i_1=(a+2b)/3$ and $i_2=(2a+b)/3$.
	Hence we obtain $a=-i_1+2i_2,b=2i_1-i_2$ and $d+1=a+b=i_1+i_2$.
	Moreover, it is easy to see that $\Lambda_{\Delta^{(3)}}=\Lambda(a,b)$ with some ordering of the vertices of $\Delta^{(3)}$.
	Hence this completes the proof of the case where $\Vol(\Delta)=3$.

	\subsection{The case $\Vol(\Delta)=4$}
	\label{sub2.3}
	In this subsection, we consider the case where $\Vol(\Delta)=4$.
	For nonnegative integers $a,b,c$,
	we let $\Lambda_1(a,b,c)$ and $\Lambda_2(a,b,c)$ be the finite abelian subgroups of $(\RR/\ZZ)^{a+b+c}$ defined as follows:
	$$\Lambda_1(a,b,c)= \left\langle \left(\underbrace{\dfrac{1}{4},\ldots,\dfrac{1}{4}}_{a},\underbrace{\dfrac{1}{2},\ldots,\dfrac{1}{2}}_{b},\underbrace{\dfrac{3}{4},\ldots,\dfrac{3}{4}}_{c}\right) \right \rangle;$$
	$$\Lambda_2(a,b,c)= \left\langle \left(\underbrace{\dfrac{1}{2},\ldots,\dfrac{1}{2}}_{a},\underbrace{\dfrac{1}{2},\ldots,\dfrac{1}{2}}_{b},\underbrace{0,\ldots,0}_{c}\right),\left(\underbrace{0,\ldots,0}_{a},\underbrace{\dfrac{1}{2},\ldots,\dfrac{1}{2}}_{b},\underbrace{\dfrac{1}{2},\ldots,\dfrac{1}{2}}_{c} \right) \right \rangle.$$
	%and if for any element of $\Lambda_1(a,b,c)$ (resp. $\Lambda_2(a,b,c)$) the sum of all entries of the element is an integer, then we let 
	%$\Delta_1(a,b,c)$ (resp. $\Delta_2(a,b,c)$) be a lattice simplex of dimension $a+b+c-1$ such that $\Lambda_{\Delta_1(a,b,c)}=\Lambda_1(a,b,c)$ (resp. $\Lambda_{\Delta_2(a,b,c)}=\Lambda_2(a,b,c)$). 
	Since $\Vol(\Delta)=|\Lambda_{\Delta}|=4$,
	for any $\lambda \in \Lambda_{\Delta}\setminus \{{\mathbf 0}\}$, $\ord(\lambda) \in \{2,4\}$. 
	Hence since $\Delta$ is not a lattice pyramid, by Lemma \ref{lem:pyr}, there exist nonnegative integers $a,b,c$ with $d+1=a+b+c$ such that $\Lambda_{\Delta}$ coincides with $\Lambda_1(a,b,c)$ or $\Lambda_2(a,b,c)$ with some ordering of the vertices of $\Delta$.
	
	First, suppose that $\Lambda_{\Delta}=\Lambda_1(a,b,c)$ with some ordering of the vertices of $\Delta$.
	Then since $\Lambda_1(a,b,c)$ coincides with $\Lambda_1(c,b,a)$ by reordering of the coordinates, we may assume that $a \geq c$.
	Moreover, by using Lemma \ref{delta}, one has $\{i_1,i_2,i_3\}=\{(a+2b+3c)/4, (a+c)/2,(3a+2b+c)/4\}$.
	Set $(h_1,h_2,h_3)=((a+2b+3c)/4, (a+c)/2,(3a+2b+c)/4)$.
	Then we obtain $a=-h_1+h_2+h_3, b=h_1-2h_2+h_3$ and $c=h_1+h_2-h_3$.
	Since $a \geq c$, 
	$(h_1,h_3)$ is $(i_1,i_2)$, $(i_1,i_3)$ or $(i_2,i_3)$.
	Hence it follows from $a,b,c \geq 0$ that one of the following conditions is satisfied:
	\begin{enumerate}
		\item $i_1+i_3 \geq 2i_2$ and $d+1=i_1+i_3$, and $\Lambda_{\Delta}=\Lambda_1(-i_1+i_2+i_3,i_1-2i_2+i_3,i_1+i_2-i_3)$;% with some ordering of the vertices of $\Delta$;
		\item $i_2+i_3 \geq 2i_1$ and $d+1=i_2+i_3$, and $\Lambda_{\Delta}=\Lambda_1(i_1-i_2+i_3,-2i_1+i_2+i_3,i_1+i_2-i_3)$;% with some ordering of the vertices of $\Delta$;
		\item  $i_1+i_2 \geq 2i_3$ and $d+1=i_1+i_2$, and  $\Lambda_{\Delta}=\Lambda_1(-i_1+i_2+i_3,i_1+i_2-2i_3,i_1-i_2+i_3)$;% with some ordering of the vertices of $\Delta$.
	\end{enumerate}
	If $i_1=i_2$ or $i_2=i_3$, then the condition $(1)$ is equivalent to one of the conditions $(2)$ and $(3)$.
	Since $i_1+i_2 \geq 2i_3$ implies that $i_1=i_2=i_3$, if the condition $(3)$ is satisfied, then condition $(2)$ is satisfied.
	Moreover, it always follows that $i_2+i_3 \geq 2i_1$.
	Hence we know that one of the following conditions is satisfied:
	\begin{enumerate}
		\item[(1')] $i_1 < i_2 < i_3$, $i_1+i_3 \geq 2i_2$ and $d+1=i_1+i_3$, and $\Lambda_{\Delta}=\Lambda_1(-i_1+i_2+i_3,i_1-2i_2+i_3,i_1+i_2-i_3)$;% with some ordering of the vertices of $\Delta$;
		\item[(2')] $d+1=i_2+i_3$ and  $\Lambda_{\Delta}=\Lambda_1(i_1-i_2+i_3,-2i_1+i_2+i_3,i_1+i_2-i_3)$.% with some ordering of the vertices of $\Delta$.
	\end{enumerate}
Now, we assume that the condition (1') is satisfied.
Then the conditions $i_1+i_3 \leq 2i_2$, $i_2 \leq \lfloor(d+1)/2 \rfloor$ and $i_1+i_3=d+1$ imply  \[2i_2=i_1+i_3=d+1.\]
This condition and the conditions $i_1<i_2<i_3$ and $i_3 \leq i_1+i_2$ satisfy  the condition of Theorem \ref{delta4} (3).
	Moreover, then, it is easy to see that  $\Lambda_{\Delta^{(4)}_1}=\Lambda_1(a,b,c)$ with some ordering of the vertices of $\Delta^{(4)}_1$.
	On the other hand, we assume that the condition (2') is satisfied.
	Then the conditions $i_3 \leq i_1+i_2$ and $i_2+i_3 =d+1$ satisfy the condition of Theorem \ref{delta4} (3).
Moreover, then, it is easy to see that $\Lambda_{\Delta^{(4)}_2}=\Lambda_1(a,b,c)$ with some ordering of the vertices of $\Delta^{(4)}_2$.

	Next, we suppose that $\Lambda_{\Delta}=\Lambda_2(a,b,c)$ with some ordering of the vertices of $\Delta$.
	It follows that $\Lambda_2(a,b,c)$ coincides with $\Lambda_2(b,a,c)$ (resp. $\Lambda_2(c,b,a)$) by reordering of the coordinates.
	Hence we may assume that $a \geq b \geq c$.
	Then by using Lemma \ref{delta},  one has $(i_1,i_2,i_3)=((b+c)/2, (a+c)/2,(a+b)/2)$.
	Therefore, we obtain $d+1=i_1+i_2+i_3$ and $\Lambda_{\Delta}=\Lambda_2(-i_1+i_2+i_3,i_1-i_2+i_3,i_1+i_2-i_3)$.
	The conditions $i_3 \leq i_1+i_2$ and $i_1+i_2+i_3 =d+1$ satisfy the condition of Theorem \ref{delta4} (3).
	Moreover, it is easy to see that $\Lambda_{\Delta^{(4)}_3}=\Lambda_2(a,b,c)$ with some ordering of the vertices of $\Delta^{(4)}_3$.
	Hence this completes the proof of the case where $\Vol(\Delta)=4$.
	
	\bigskip
	Therefore, Theorem \ref{simplex} follows.
	\section{Proof of Theorem \ref{nonspan}}
	\label{sec3}
	In this section, we prove Theorem \ref{nonspan}.
	First, we show the following.
	\begin{Lemma}
		One has
			$\delta(\Ac^{(4)}_i,t)=1+t+t^k+t^{k+1}$ for $1 \leq i \leq 3$ and
			 $\delta(\Bc^{(4)},t)=1+t+2t^{k}$.
	\end{Lemma}
\begin{proof}
We show that $\delta(\Bc^{(4)},t)=1+t+2t^{k}$.
Let $\Delta_1=\conv({\bf 0}, \eb_1,\ldots,\eb_{d-1},\eb_2+\cdots+\eb_{d-1}+2\eb_d)$ and
$\Delta_2=\conv( \eb_1,\ldots,\eb_{d-1},\eb_2+\cdots+\eb_{d-1}+2\eb_d,\eb_1+\eb_2)$.
Then $\Bc^{(4)}$ has a lattice triangulation which consists of two maximal simplices $\Delta_1$ and $\Delta_2$.
Indeed, the hyperplane 
\[
\Hc=\{(x_1,\ldots,x_d) \in \RR^d : 2(x_1+\cdots+x_{d-1})-(2k-1)x_d=2 \}
\]
divides $\Bc^{(4)}$ into $\Delta_1$ and $\Delta_2$.
Moreover, one has $\delta(\Delta_1,t)=\delta(\Delta_2,t)=1+t^k$
and $\delta(\Delta_1 \cap \Delta_2,t)=1$. 
Hence by using Lemmas \ref{delta_ehr} and \ref{delta_split}, we obtain
\begin{equation*}
\begin{split}
L_{\Bc^{(4)}}(n)&=\binom{n+d}{d}+2\binom{n+d-k}{d}+\left(\binom{n+d}{d}-\binom{n+d-1}{d-1}\right)	\\
&=\binom{n+d}{d}+2\binom{n+d-k}{d}+\binom{n+d-1}{d}.
\end{split}
\end{equation*}
Therefore, one has 	 $\delta(\Bc^{(4)},t)=1+t+2t^{k}$.
The remaining cases can be shown by the same way.

\end{proof}
	Let $\Pc \subset \RR^d$ be a lattice non-spanning non-simplex of dimension $d$ with $\Vol(\Pc)=4$.
	Assume that $\Pc$ is not a lattice pyramid.
	Finally, we show that $\Pc$ is unimodularly equivalent to one of the lattice polytopes in Theorem \ref{nonspan}.
	We consider the affine lattice $L_{\Pc}$ generated by $\Pc \cap \ZZ^d$.
	Let $\tilde{\Pc} \subset \RR^d$ be the lattice spanning polytope given by the vertices of $\Pc$ with respect to the lattice $L_{\Pc}$.
	Since $\Pc$ is not a simplex and is not spanning,  the normalized volume of $\tilde{\Pc}$ is equal to $2$.
	Hence by Theorem \ref{spanning}, $\tilde{\Pc}$ is unimodularly equivalent to a lattice pyramid of a unit square.
	In particular, four vertices $\tilde{\vb}_1,\ldots,\tilde{\vb}_4$ of $\tilde{\Pc}$ satisfy the relation $\tilde{\vb}_1+\tilde{\vb}_2=\tilde{\vb}_3+\tilde{\vb}_4$.
	Therefore, $\Pc$ can be determined as $\Pc=\conv(S \cup \{\ub\})$, for an empty simplex $S$ of normalized volume $2$, and a lattice point $u=(u_1,\ldots,u_d)$ satisfying $\ub=\vb_1+\vb_2-\vb_3$, for some $\vb_i$ chosen among the vertices of $S$.
By Theorem \ref{simplex}, we can assume that 
\[S=\conv({\bf 0}, \eb_1,\ldots,\eb_{d-1}, \eb_{a+1}+\cdots+\eb_{d-1}+2\eb_d)
\]
for some integer $a \geq 0$.
If $a \geq 4$, then for some $1 \leq i \leq a$, 
$u_i=0$. This implies that $\Pc$ is a lattice pyramid.
Hence one has $a \leq 3$.
\subsection{The case $a=3$}
If $d$ is odd, then $S$ is unimodularly equivalent to 
\[
\conv({\bf 0}, \eb_1,\ldots,\eb_{d-1}, \eb_{a+2}+\cdots+\eb_{d-1}+2\eb_d).
\]
Hence we can assume that $d$ is even.
Since $\Pc$ is not a lattice pyramid, we obtain $\{\vb_1,\vb_2,\vb_3 \} = \{\eb_1,\eb_2,\eb_3\}$.
Therefore, we may suppose that $\ub=(1,1,-1,0,\ldots,0)$.
Thus $\Pc=\Ac^{(4)}_3$.
\subsection{The case $a=2$}
Similarly, we can assume that $d$ is odd.
Moreover, since $\Pc$ is not a lattice pyramid,
we obtain $\{\vb_1,\vb_2,\vb_3 \} \supset \{\eb_1,\eb_2\}$.
Hence we may suppose that $\ub=\eb_1+\eb_2-\vb$ or $\eb_1-\eb_2+\vb$, where $\vb \in \{{\bf 0}, \eb_3,\eb_3+\cdots+\eb_{d-1}+2\eb_d \}$.
For any case, $\Pc$ is unimodularly equivalent to $\Ac_{2}^{(4)}$.
\subsection{The case $a=1$}
Similarly, we can assume that $d$ is even.
Moreover, since $\Pc$ is not a lattice pyramid,
we obtain $\eb_1 \in \{\vb_1,\vb_2,\vb_3 \}$.
Hence we may suppose that $\ub=\eb_1+\ub_1-\ub_2$ or $-\eb_1+\ub_1+\ub_2$, where $\ub_1,\ub_2 \in \{{\bf 0}, \eb_2,\eb_3,\eb_2+\cdots+\eb_{d-1}+2\eb_d \}$.
When $\ub=\eb_1+\ub_1-\ub_2$, $\Pc$ is unimodularly equivalent to $\Bc^{(4)}$.
When $\ub=-\eb_1+\ub_1+\ub_2$, $\Pc$ is unimodularly equivalent to $\Ac^{(4)}_1$.
\subsection{The case $a=0$}
Similarly, we can assume that $d$ is odd.
Moreover, we may suppose that $\eb_1 \in \{\vb_1,\vb_2,\vb_3\}$.
When $\vb_1=\eb_1$, we should consider $(\vb_2,\vb_3)=(\eb_2,\eb_3)$ or $\{\vb_2,\vb_3\} \subset \{{\bf 0}, \eb_1,\eb_1+\cdots+\eb_{d-1}+2\eb_d\}$.
For each case, $\Pc$ is unimodularly equivalent to $\Ac^{(4)}_2$.
When $\vb_3=\eb_1$, we should consider only the case $(\vb_1,\vb_2)=({\bf 0}, \eb_1+\cdots+\eb_{d-1}+2\eb_d)$.
In this case, $\Pc$ is also unimodularly equivalent to $\Ac^{(4)}_2$.

Therefore, we complete the proof of Theorem \ref{nonspan}.
	
\end{document}